\documentclass[11pt,reqno]{amsart}
\usepackage{amsmath, amssymb, amsthm}
\usepackage{url}
\usepackage{hyperref}
\usepackage{tikz}
\usepackage{ytableau}
\usepackage[utf8]{inputenc}
\usepackage{xcolor}
\usepackage{tikz-cd}
\usepackage{hyperref}
\usepackage{nccmath}

\usepackage{youngtab}

\usetikzlibrary{decorations.pathreplacing}

\usetikzlibrary{decorations.pathreplacing}

\newcommand{\p}{\overline{p}}

\renewcommand{\(}{\left\(}
\renewcommand{\)}{\right\)}
\renewcommand{\[}{\left\[}
\renewcommand{\]}{\right\]}

\numberwithin{equation}{section}
\theoremstyle{plain}
\newtheorem{theorem}{Theorem}[section]
\newtheorem{lemma}[theorem]{Lemma}

\newtheorem{definition}[theorem]{Definition}

\newtheorem{problem}[theorem]{Problem}

\newtheorem{corollary}[theorem]{Corollary}
\newtheorem{proposition}[theorem]{Proposition}

\makeatletter
\def\proof{\@ifnextchar[{\@oproof}{\@nproof}}
\def\@oproof[#1][#2]{\trivlist\item[\hskip\labelsep\textit{#2 Proof of\
		#1.}~]\ignorespaces}
\def\@nproof{\trivlist\item[\hskip\labelsep\textit{Proof.}~]\ignorespaces}

\makeatother

\begin{document}
	
	\title[Higher order differences of the logarithm of the overpartition function]{Inequalities for higher order differences of the logarithm of the overpartition function and a problem of Wang-Xie-Zhang} 
	
	\author{Gargi Mukherjee}
	\address{Institute for Algebra, Johannes Kepler University, Altenberger Straße 69, A-4040 Linz, Austria.}
	\email{gargi.mukherjee@dk-compmath.jku.at}
	
	\maketitle
	
	\begin{abstract}
Let $\overline{p}(n)$ denote the overpartition function. In this paper, our primary goal is to study the asymptotic behavior of the finite differences of the logarithm of the overpartition function, i.e., $(-1)^{r-1}\Delta^r \log \p(n)$, by studying the inequality of the following form 
$$\log \Bigl(1+\dfrac{C(r)}{n^{r-1/2}}-\dfrac{C_1(r)}{n^{r}}\Bigr)<(-1)^{r-1}\Delta^r \log \p(n) <\log \Bigl(1+\dfrac{C(r)}{n^{r-1/2}}\Bigr)\ \text{for}\ n \geq N(r),$$
where $C(r), C_1(r), \text{and}\ N(r)$ are computable constants depending on the positive integer $r$, determined explicitly. This solves a problem posed by Wang, Xie and Zhang in the context of searching for a better lower bound of $(-1)^{r-1}\Delta^r \log \p(n)$ than $0$. By settling the problem, we are able to show that 
\begin{equation*}
\lim_{n\rightarrow \infty}(-1)^{r-1}\Delta^r \log \p(n) =\dfrac{\pi}{2}\Bigl(\dfrac{1}{2}\Bigr)_{r-1}n^{\frac{1}{2}-r}.
\end{equation*}
	\end{abstract}

\hspace{0.65 cm} \textbf{Mathematics Subject Classifications.} Primary 05A20; 11B68.\\
\vspace{0.3 cm}
\hspace{0.8 cm} \textbf{Keywords.} Overpartition, $\log$-concavity, Finite difference.

\section{Introduction}\label{intro}
An overpartition of a positive integer $n$ is a nonincreasing sequence of positive integers whose sum is $n$ in which the first occurrence of a number
may be overlined, $\overline{p}(n)$ denotes the number of overpartitions of $n$, and we define $\overline{p}(0)=1$. For example, there are $8$ overpartitions of $3$ enumerated by $3, \overline{3}, 2+1, \overline{2}+1, 2+\overline{1}, \overline{2}+\overline{1},1+1+1, \overline{1}+1+1$. A thorough study of the overpartition function started with the work of Corteel and Lovejoy \cite{CorteelLovejoy}, although it has been studied under different nomenclature that dates back to MacMahon. Similar to the Hardy-Ramanujan-Rademacher formula for the partition function (cf. \cite{RamanujanHardy},\cite{Rademacher}), Zuckerman's \cite{Zuckerman} formula for $\overline{p}(n)$ states that
\begin{equation}\label{Zuckerman}
\overline{p}(n)=\frac{1}{2\pi}\underset{2 \nmid k}{\sum_{k=1}^{\infty}}\sqrt{k}\underset{(h,k)=1}{\sum_{h=0}^{k-1}}\dfrac{\omega(h,k)^2}{\omega(2h,k)}e^{-\frac{2\pi i n h}{k}}\dfrac{d}{dn}\biggl(\dfrac{\sinh \frac{\pi \sqrt{n}}{k}}{\sqrt{n}}\biggr),
\end{equation}
where
\begin{equation*}
\omega(h,k)=\text{exp}\Biggl(\pi i \sum_{r=1}^{k-1}\dfrac{r}{k}\biggl(\dfrac{hr}{k}-\biggl\lfloor\dfrac{hr}{k}\biggr\rfloor-\dfrac{1}{2}\biggr)\Biggr)
\end{equation*}
for positive integers $h$ and $k$. Similarly as Lehmer \cite{Lehmer} obtained an error bound for the partition function $p(n)$, Engel \cite{Engel} determined an error term for $\overline{p}(n)$ and found that
\begin{equation}\label{Engel1}
\overline{p}(n)=\frac{1}{2\pi}\underset{2 \nmid k}{\sum_{k=1}^{N}}\sqrt{k}\underset{(h,k)=1}{\sum_{h=0}^{k-1}}\dfrac{\omega(h,k)^2}{\omega(2h,k)}e^{-\frac{2\pi i n h}{k}}\dfrac{d}{dn}\biggl(\dfrac{\sinh \frac{\pi \sqrt{n}}{k}}{\sqrt{n}}\biggr)+R_{2}(n,N),
\end{equation}
where
\begin{equation}\label{Engel2}
\bigl|R_{2}(n,N)\bigr|< \dfrac{N^{5/2}}{\pi n^{3/2}} \sinh \biggl(\dfrac{\pi \sqrt{n}}{N}\biggr).
\end{equation}
A positive sequence $\{a_n\}_{n \geq 0}$ is $\log$-concave if for all $n \geq 1$,
\begin{equation*}
a^2_n-a_{n-1}a_{n+1} \geq 0.
\end{equation*} 
Engel \cite{Engel} proved that $\{\overline{p}(n)\}_{n \geq 2}$ is $\log$-concave by using the asymptotic formula \eqref{Engel1} with $N=3$ followed by \eqref{Engel2}. Prior to Engel's work on overpartitions, the $\log$-concavity of the partition function $p(n)$ and its associated inequalities has been studied in a wider spectrum, details can be found in \cite{Chentalk}, \cite{Chen2}, and \cite{DeSalvoPak}. Liu and Zhang \cite{LiuZhang} proved a family of inequalities for the overpartition function.\\
Chen, Guo and Wang \cite{Chen3} introduced the notion of ratio $\log$-convexity of a sequence and established that ratio $\log$-convexity implies $\log$-convexity under a certain initial condition. A sequence $\{a_n\}_{n \geq k}$ is called ratio $\log$-convex if $\{a_{n+1}/a_n\}_{n \geq k}$ is $\log$-convex or, equivalently, for $n \geq k+1$,
$$\log a_{n+2}-3\log a_{n+1}+3\log a_n-\log a_{n-1}\geq 0.$$
Let $\Delta$ be the difference operator defined by $\Delta f(n)=f(n+1)-f(n)$. Similar to the work done by Chen et al. \cite{Chen2} for $p(n)$, Wang, Xie and Zhang \cite{WangXieZhang} proved the following two theorems.
\begin{theorem}\cite[Theorem 3.1]{WangXieZhang}\label{wxz1}
For each $r\geq1$, there exists a positive number $n(r)$ such that for all $n \geq n(r)$, $$(-1)^{r-1}\Delta^r\log\overline{p}(n)>0.$$	
\end{theorem}
\begin{theorem}\cite[Theorem 4.1]{WangXieZhang}\label{wxz2}
	For each $r\geq1$, there exists a positive number $n(r)$ such that for all $n \geq n(r)$, $$(-1)^{r-1}\Delta^r\log\overline{p}(n)<1+\dfrac{\pi}{2}\Bigl(\dfrac{1}{2}\Bigr)_{r-1}\dfrac{1}{n^{r-\frac{1}{2}}},$$	
\end{theorem}
where $\left(\alpha\right)_{r}:=\alpha\cdot(\alpha+1)\cdots(\alpha+n-1)$.

They raised the following question:
\begin{problem}\cite[Problem 3.4]{WangXieZhang}\label{wxz3}
Does there exist a positive number A such that
\begin{equation*}
\dfrac{(-1)^{r-1}\Delta^r \log \p(n)}{n^{-(r-1/2)}}>A,
\end{equation*}
for any $r$ and all sufficiently large $n$?
\end{problem}
In other words, their statement reads
 ``\textit{Moreover, we also wish to seek for a sharp lower bound for $(-1)^{r-1}\Delta^r \log \p(n)$}''. \\
The main motivation of this paper is to give an affirmative solution to the Problem \ref{wxz3} in Theorems \ref{mainresult1} and \ref{mainresult2}. This in turn clarifies the asymptotic growth of $(-1)^{r-1}\Delta^r \log \p(n)$, see Corollary \ref{Cor1}. Moreover, we reprove the $\log$-concavity and its companion inequality in Corollary \ref{Cor2}.

\begin{theorem}\label{mainresult1}
For $n \geq 26$,
\begin{equation}\label{mainresult1eqn}
\log \Bigl(1+\dfrac{\pi}{2\sqrt{n}}\Bigr)<\Delta \log \p(n)<\log \Bigl(1+\dfrac{\pi}{2\sqrt{n}}+\dfrac{\pi^2}{40n}\Bigr).
\end{equation}
\end{theorem}

For $r \geq 2$, we define
\begin{definition}\label{Def}
	\begin{eqnarray}\label{def1}
		N_0(m) &:=&
	\begin{cases}
	1, &\quad \text{if}\ m=1,\\
	2m \log m-m \log \log m, & \quad \text{if}\ m \geq 2,
	\end{cases}\\\label{def2}
	N_1(r)&:=&\max\Biggl\{85, \Biggl\lceil \dfrac{4}{\pi^2}N^2_0(2r+2)\Biggr\rceil \Biggr\},\\\label{def3}
	C(r)&:=&\dfrac{\pi}{2}\Bigl(\dfrac{1}{2}\Bigr)_{r-1},\\\label{def4}
	C_1(r)&:=&(r-1)!+4r^2 C(r),\\\label{def5}
	C_2(r)&:=& \sum_{k=0}^{2r-2}\dfrac{1}{(k+1)\pi^{k+1}}\Bigl(\dfrac{k+1}{2}\Bigr)_r \dfrac{1}{r^k}+\dfrac{r}{10^r},\\\label{def6}
	N_2(r)&:=&\Biggl \lceil \Biggl(\dfrac{1+C_1(r)}{C(r)}\Biggr)^2\Biggr\rceil,\\\label{def7}
	N_3(r)&:=&\max \Biggl\{N_1(r),2r^2, \Biggl \lceil \Biggl(\dfrac{2^{r+1}\Bigl(C_2(r)+1\Bigr)}{(r-1)!}\Biggr)^2\Biggr \rceil, \Biggl\lceil \sqrt[r-1]{\Biggl(\dfrac{2^r C^2(r)}{(r-1)!}\Biggr)} \Biggr\rceil \Biggr\},\nonumber\\
	\and \hspace{3 cm} & & \\\label{def8}
	N(r)&:=&\max \Bigl\{N_2(r), N_3(r)\Bigr\}.
	\end{eqnarray}
\end{definition}

\begin{theorem}\label{mainresult2}
For $r \in \mathbb{Z}_{\geq 2}$ and $n \geq N(r)$,
\begin{equation}\label{mainresult2eqn}
0<\log \Bigl(1+\dfrac{C(r)}{n^{r-1/2}}-\dfrac{C_1(r)}{n^{r}}\Bigr)<(-1)^{r-1}\Delta^r \log \p(n) <\log \Bigl(1+\dfrac{C(r)}{n^{r-1/2}}\Bigr),
\end{equation}
where $C(r)$ and $C_1(r)$ are given in \eqref{def3}-\eqref{def4}.
\end{theorem}

\begin{corollary}\label{Cor1} For $r \in \mathbb{Z}_{\geq 1}$,
	\begin{equation}\label{coreqn1}
	\lim_{n\rightarrow \infty} n^{r-1/2}(-1)^{r-1}\Delta^r \log \p(n)= \dfrac{\pi}{2}\Bigl(\dfrac{1}{2}\Bigr)_{r-1}.
	\end{equation}
\end{corollary}
\begin{proof}
Multiplying both sides of \eqref{mainresult1eqn} (resp. \eqref{mainresult2eqn}) by $\sqrt{n}$ (resp. by $n^{r-1/2}$) and taking limit as $n$ tends to infinity, we obtain \eqref{coreqn1}.
\end{proof}

\begin{corollary}\cite[Theorem 1.2]{Engel}\label{Cor2} For $n \geq 4$, $\p(n)$ is $\log$-concave.
\end{corollary}
\begin{proof}
Observe that $N(2)=344$ and from the lower bound of \eqref{mainresult2eqn}, we observe that $\{\p(n)\}_{n \geq 344}$ is $\log$-concave and for the rest $5 \leq n \leq 343$, we confirm by numerical checking in Mathematica.
\end{proof}

\begin{corollary}\cite[Equation (1.6)]{LiuZhang}\label{Cor3} For $n \geq 2$, 
	\begin{equation}\label{coreqn2}
	\dfrac{\p(n-1)}{\p(n)}\Bigl(1+\dfrac{\pi}{4n^{3/2}}\Bigr)>\dfrac{\p(n+1)}{\p(n)}.
	\end{equation}
\end{corollary}
\begin{proof}
 Similar to the proof of Corollary \ref{Cor2}, take $r=2$ and from the upper bound of \eqref{mainresult2eqn}, we conclude the proof.
\end{proof}

\begin{corollary}\label{Cor4} For $n \geq 18$, $\p(n)$ is ratio $\log$-convex. 
\end{corollary}
\begin{proof}
Take $r=3$ and observe that $N(3)=1486$ and rest of the proof is similar to the proof of Corollary \ref{Cor2}..
\end{proof}

As an immediate consequence of Corollary \ref{Cor4}, we have

\begin{corollary}\cite[Corollary 1.3]{Gargi}\label{Cor5}
The sequence $\{\sqrt[n]{\p(n)}\}_{n \geq 4}$ is $\log$-convex.	
\end{corollary}

This paper is organized as follows. A preliminary setup for decomposing $(-1)^{r-1}\Delta^r \log \p(n)$ $=H_r+G_r$ (cf. see \eqref{eqn4} and \eqref{eqn5}), as done in \cite{WangXieZhang} and consequently, estimations for both $H_r$ and $G_r$ are given in Section \ref{sec2}. Proofs of Theorems \ref{mainresult1} and \ref{mainresult2} are given in Section \ref{sec3}.
\section{preliminary lemmas}\label{sec2}
Following the notations given in Engel \cite{Engel} and Wang, Xie and Zhang \cite{WangXieZhang}, split $\p(n)$ as
\begin{equation}\label{eqn1}
\p(n)=\widehat{T}(n) \Biggl(1+\dfrac{\widehat{R}(n)}{\widehat{T}(n)}\Biggr),
\end{equation}
where 
\begin{eqnarray}\label{eqn2}
\widehat{T}(n)&=& \dfrac{1}{8n}\Bigl(1-\dfrac{1}{\widehat{\mu}(n)}\Bigr)e^{\widehat{\mu}(n)}\\ \label{eqn3}
\text{and}\ \ \widehat{R}(n)&=& \dfrac{1}{8n}\Bigl(1+\dfrac{1}{\widehat{\mu}(n)}\Bigr)e^{-\widehat{\mu}(n)}+R_2(n,3)
\end{eqnarray}
with $\widehat{\mu}(n)=\pi \sqrt n$.\\
Taking the logarithm on both sides of \eqref{eqn1} and plugging the definitions from \eqref{eqn2}-\eqref{eqn3}, we obtain
\begin{equation*}
\log \p(n)=\log \dfrac{\pi^2}{8}-3 \log \widehat{\mu}(n)+\log (\widehat{\mu}(n)-1)+\widehat{\mu}(n)+\log\  \Biggl(1+\dfrac{\widehat{R}(n)}{\widehat{T}(n)}\Biggr).
\end{equation*}
Therefore, 
\begin{equation}\label{eqn4}
(-1)^{r-1}\Delta^r \log \p(n)=H_r+G_r,
\end{equation}
where
\begin{eqnarray}\label{eqn5}
H_r &=& (-1)^{r-1}\Delta^r (-3 \log \widehat{\mu}(n)+\log (\widehat{\mu}(n)-1)+\widehat{\mu}(n))\\ \label{eqn6}
G_r &=& (-1)^{r-1}\Delta^r \log\  \Biggl(1+\dfrac{\widehat{R}(n)}{\widehat{T}(n)}\Biggr).
\end{eqnarray}
Then we have that for $r \geq 1$,
\begin{equation}\label{eqn8}
H_r-|G_r| \leq (-1)^{r-1}\Delta^r \log \p(n) \leq H_r+|G_r|.
\end{equation}
To estimate the bounds for $(-1)^{r-1}\Delta^r \log \p(n)$, we need to establish bounds for $H_r$ and $|G_r|$. Our first goal is to determine a bound for $|G_r|$ for $r \geq 1$ and then we further proceed with $H_r$ but splitting in two cases, i.e., for $r=1$ and $r \geq 2$.
\begin{lemma}\cite[Lemma 2.1]{GZZ}\label{lem0}
	For any integer $m\geq1$ and $x \geq N_0(m)$,
	\begin{align*}
	x^me^{-x}<1,
	\end{align*}
where $N_0(m)$ is defined in \eqref{def1}.
\end{lemma}
Recall that $N_1(r)= \max\Biggl\{85, \Biggl\lceil \dfrac{4}{\pi^2}N^2_0(2r+2)\Biggr\rceil \Biggr\}$ (cf. \eqref{def2}).
\begin{lemma}\label{lem1}
For all $n \geq N_1(r)$ and $r \geq 1$,
\begin{equation}\label{lem1eqn1}
|G_r| < \dfrac{1}{n^{r+1}}.
\end{equation}
\end{lemma} 
\begin{proof}
Define $\widehat{e}(n):= \frac{\widehat{R}(n)}{\widehat{T}(n)}$. From the definition of $\widehat{R}(n)$ and $\widehat{T}(n)$ (cf. Equation \eqref{eqn2}-\eqref{eqn3}), we have
\begin{eqnarray}\label{lem1eqn3}
|\widehat{e}(n)|&=&\dfrac{|\widehat{R}(n)|}{|\widehat{T}(n)|}\nonumber\\
&=& \Biggl|\dfrac{\frac{1}{8n}\Bigl(1+\frac{1}{\widehat{\mu}(n)}\Bigr) e^{-\widehat{\mu}(n)}+R_2(n,3)}{\frac{1}{8n}\Bigl(1-\frac{1}{\widehat{\mu}(n)}\Bigr) e^{\widehat{\mu}(n)}}\Biggr|\nonumber\\
&<& \dfrac{\widehat{\mu}(n)+1}{\widehat{\mu}(n)-1}e^{-2\widehat{\mu}(n)}+ \dfrac{36\sqrt{3}}{\widehat{\mu}(n)-1}e^{-2\widehat{\mu}(n)/3}\nonumber\\
& & \hspace{4 cm} \ \ \Bigl(\text{using}\ N=3\ \text{in}\ \eqref{Engel2}\ \text{and}\ \sinh(x)<\frac{e^x}{2}\ \text{for}\ x>0\Bigr)\nonumber\\
&=& \dfrac{1}{\widehat{\mu}(n)-1}e^{-\widehat{\mu}(n)/2} \Bigl((\widehat{\mu}(n)+1)e^{-3\widehat{\mu}(n)/2}+36\sqrt{3}\ e^{-\widehat{\mu}(n)/6}\Bigr).
\end{eqnarray}

Since for all $n \geq 85$, $$(\widehat{\mu}(n)+1)e^{-3\widehat{\mu}(n)/2}+36\sqrt{3}\ e^{-\widehat{\mu}(n)/6}<\frac{1}{2}\ \text{and}\  \frac{1}{\widehat{\mu}(n)-1}<1,$$ 
from \eqref{lem1eqn3}, it follows that for all $n \geq 85$,
\begin{equation}\label{lem1eqn4}
|\widehat{e}(n)|<\dfrac{1}{2}\ e^{-\widehat{\mu}(n)/2}.
\end{equation}
Therefore, for all $n \geq 85$,
\begin{eqnarray}\label{lem1eqn5}
|G_r| &=& \Bigl|\sum_{i=0}^{r}(-1)^{r-1}\Delta^r \log\ (1+\widehat{e}(n))\Bigr| \ \ (\text{by}\  \eqref{eqn6})\nonumber\\
&=& \Biggl|\sum_{i=0}^{r}(-1)^{r-i}\binom{r}{i}\log\ (1+\widehat{e}(n+i))\Biggr|\nonumber\\
&\leq& \sum_{i=0}^{r} \binom{r}{i} \Bigl|\log\ (1+\widehat{e}(n+i))\Bigr|\nonumber\\
&\leq& \sum_{i=0}^{r} \binom{r}{i} \dfrac{|\widehat{e}(n+i)|}{1-|\widehat{e}(n+i)|}\ \ \Bigl(\text{since},\  |\log (1+x)|\leq \dfrac{|x|}{1-|x|}\ \text{for}\ |x|<1\Bigr)\nonumber\\
& \leq & 2\sum_{i=0}^{r} \binom{r}{i} |\widehat{e}(n+i)|\ \Bigl(\text{as}\ \dfrac{x}{1-x}\leq 2x\ \text{for}\ 0<x\leq\frac{1}{2}\Bigr)\nonumber\\
& < & \sum_{i=0}^{r} \binom{r}{i} e^{-\widehat{\mu}(n+i)/2} \ (\text{by}\ \eqref{lem1eqn4})\nonumber\\
& \leq & \sum_{i=0}^{r} \binom{r}{i} e^{-\widehat{\mu}(n)/2} \ \Bigl(\text{since},\ \{e^{-\widehat{\mu}(n)/2}\}_{n \geq 1}\ \text{is a decreasing sequence}\Bigr)\nonumber\\
&=& 2^r e^{-\widehat{\mu}(n)/2}.
\end{eqnarray}
Now applying Lemma \ref{lem0} with $m=2r+2$ and assigning $x \mapsto \frac{\widehat{\mu}(n)}{2}$, it follows that for $n \geq \Biggl\lceil \dfrac{4}{\pi^2}N^2_0(2r+2)\Biggr\rceil$,
\begin{equation}\label{lem1eqn7}
 e^{-\widehat{\mu}(n)/2}< \Bigl(\dfrac{2}{\pi}\Bigr)^{2r+2}\dfrac{1}{n^{r+1}} \implies  2^{r} e^{-\widehat{\mu}(n)/2} < \Bigl(\dfrac{2\sqrt{2}}{\pi}\Bigr)^{2r+2}\dfrac{1}{n^{r+1}}<\dfrac{1}{n^{r+1}}.
\end{equation}
\end{proof}

Before we state the bounds for $H_r$, we recall the following result due to Odlyzko \cite{Odlyzko} on the relation between the higher order differences of a smooth function and its derivatives.
\begin{proposition}\label{Odprop}
Let $r$ be a positive integer. Suppose that $f(x)$ is a function with infinite continuous derivatives for $x\geq 1$, and $(-1)^{k-1}f^{(k)}(x)>0$ for $k \geq 1$. Then for $r\geq 1$,
$$(-1)^{r-1}f^{(r)}(x+r)\leq (-1)^{r-1}\Delta^r f(x) \leq (-1)^{r-1}f^{(r)}(x).$$	
\end{proposition}

\begin{lemma}\label{lem2}
For all $n \geq 1$,
\begin{equation}\label{lem2eqn1}
L^{(1)}(n) \leq H_1 \leq U^{(1)}(n),
\end{equation}
where 
\begin{eqnarray}\label{lem2eqn2}
U^{(1)}(n) &=& \dfrac{\pi}{2\sqrt{n}}-\dfrac{3}{2(n+1)}+\dfrac{\pi}{2\sqrt{n}(\widehat{\mu}(n)-1)}\\ \label{lem2eqn3}
\text{and}\  L^{(1)}(n) &=& \dfrac{\pi}{2\sqrt{n+1}}-\dfrac{3}{2n}+\dfrac{\pi}{2\sqrt{n+1}(\widehat{\mu}(n+1)-1)}.
\end{eqnarray}
\end{lemma} 
\begin{proof}
Equation \eqref{lem2eqn1} follows immediately by applying Proposition \ref{Odprop} on each of the factors in $H_r$ being present in \eqref{eqn5} for $r=1$.
\end{proof}

\begin{lemma}\label{lem3}
For $r \geq 2$ and $n \geq 2r^2$,
\begin{equation}\label{lem3eqn1}
\dfrac{C(r)}{n^{r-\frac{1}{2}}}-\dfrac{C_1(r)}{n^{r}}< H_r < \dfrac{C(r)}{n^{r-\frac{1}{2}}}-\dfrac{(r-1)!}{2^r\ n^r}+\dfrac{C_2(r)}{n^{r+\frac{1}{2}}},
\end{equation}
where
$C(r)$, $C_1(r)$, and $C_2(r)$ are given by \eqref{def3}-\eqref{def5}.
\end{lemma}
\begin{proof}
Rewrite \eqref{eqn5} as 
\begin{equation}\label{lem3eqn2}
H_r=(-1)^{r-1}\Delta^r (\widehat{\mu}(n)-2\log \widehat{\mu}(n)) -\sum_{k=1}^{\infty} (-1)^{r-1}\Delta^r \Bigl(\dfrac{1}{k \widehat{\mu}(n)^k}\Bigr)
\end{equation}
and applying Proposition \ref{Odprop}, we get
\begin{equation}\label{lem3eqn3}
\begin{split}
\dfrac{\pi}{2}\Bigl(\dfrac{1}{2}\Bigr)_{r-1}\dfrac{1}{(n+r)^{r-\frac{1}{2}}}-\dfrac{(r-1)!}{n^r}&+\sum_{k=1}^{\infty}\dfrac{1}{k\pi^k}\Bigl(\dfrac{k}{2}\Bigr)_r \dfrac{1}{(n+r)^{r+\frac{k}{2}}}\leq H_r\\
& \leq \dfrac{\pi}{2}\Bigl(\dfrac{1}{2}\Bigr)_{r-1}\dfrac{1}{n^{r-\frac{1}{2}}}-\dfrac{(r-1)!}{(n+r)^r}+\sum_{k=1}^{\infty}\dfrac{1}{k\pi^k}\Bigl(\dfrac{k}{2}\Bigr)_r \dfrac{1}{n^{r+\frac{k}{2}}}.
\end{split}
\end{equation}
Since for all positive integers $n$, $r$ and $k$, $$\sum_{k=1}^{\infty}\dfrac{1}{k\pi^k}\Bigl(\dfrac{k}{2}\Bigr)_r \dfrac{1}{(n+r)^{r+\frac{k}{2}}}>0.$$
Therefore, 
\begin{equation}\label{lem3eqn4}
\dfrac{\pi}{2}\Bigl(\dfrac{1}{2}\Bigr)_{r-1}\dfrac{1}{(n+r)^{r-\frac{1}{2}}}-\dfrac{(r-1)!}{n^r}< H_r \leq \dfrac{\pi}{2}\Bigl(\dfrac{1}{2}\Bigr)_{r-1}\dfrac{1}{n^{r-\frac{1}{2}}}-\dfrac{(r-1)!}{(n+r)^r}+\sum_{k=1}^{\infty}\dfrac{1}{k\pi^k}\Bigl(\dfrac{k}{2}\Bigr)_r \dfrac{1}{n^{r+\frac{k}{2}}}.
\end{equation}
Now we further investigate the lower bound of $H_r$, given in \eqref{lem3eqn4}.
\begin{eqnarray}\label{lem3eqn5}
H_r &\geq & \dfrac{\pi}{2}\Bigl(\dfrac{1}{2}\Bigr)_{r-1}\dfrac{1}{(n+r)^{r-\frac{1}{2}}}-\dfrac{(r-1)!}{n^r}\nonumber\\
&=& \dfrac{\pi}{2}\Bigl(\dfrac{1}{2}\Bigr)_{r-1}\dfrac{1}{n^{r-\frac{1}{2}}}\Bigl(1+\dfrac{r}{n}\Bigr)^{-r+\frac{1}{2}}-\dfrac{(r-1)!}{n^r}\nonumber\\
&=& \dfrac{\pi}{2}\Bigl(\dfrac{1}{2}\Bigr)_{r-1}\dfrac{1}{n^{r-\frac{1}{2}}}+\dfrac{\pi}{2}\Bigl(\dfrac{1}{2}\Bigr)_{r-1}\dfrac{1}{n^{r-\frac{1}{2}}} \sum_{m=1}^{\infty}\binom{-\frac{2r-1}{2}}{m}\Bigl(\dfrac{r}{n}\Bigr)^m-\dfrac{(r-1)!}{n^r}.
\end{eqnarray}
To bound the infinite series in \eqref{lem3eqn5}, we proceed as follows
\begin{eqnarray}\label{lem3eqn6}
\Biggl|\sum_{m=1}^{\infty}\binom{-\frac{2r-1}{2}}{m}\Bigl(\dfrac{r}{n}\Bigr)^m\Biggr| &=& \Biggl|\sum_{m=1}^{\infty}\dfrac{(-1)^m}{4^m}\dfrac{\binom{2r+2m-2}{r+m-1}\binom{r+m-1}{r-1}}{\binom{2r-2}{r-1}}\Bigl(\dfrac{r}{n}\Bigr)^m\Biggr|\nonumber\\
&\leq& \sum_{m=1}^{\infty} \dfrac{1}{4^m} \dfrac{\binom{2r+2m-2}{r+m-1}\binom{r+m-1}{r-1}}{\binom{2r-2}{r-1}}\Bigl(\dfrac{r}{n}\Bigr)^m \nonumber\\
& \leq & \sum_{m=1}^{\infty} \dfrac{2\sqrt{r-1}}{\sqrt{\pi(r+m-1)}} \binom{r+m-1}{r-1}\Bigl(\dfrac{r}{n}\Bigr)^m\nonumber\\
& & \hspace{3 cm} \Biggl(\text{since},\ \dfrac{4^k}{2\sqrt{k}}\leq \binom{2k}{k}\leq \dfrac{4^k}{\sqrt{\pi k}}\ \forall\ k \geq 1\Biggr)\nonumber\\
& < &  \dfrac{2r}{n}\sum_{m=0}^{\infty} \binom{r+m}{r-1}\Bigl(\dfrac{r}{n}\Bigr)^m\nonumber\\
&\leq & \dfrac{2r}{n}\sum_{m=0}^{\infty}r^{m+1}\Bigl(\dfrac{r}{n}\Bigr)^m\ \Biggl(\text{as},\ \binom{r+m}{r-1}\leq r^{m+1}\Biggr)\nonumber\\
&=& \dfrac{2r^2}{n}\sum_{m=0}^{\infty}\Bigl(\dfrac{r^2}{n}\Bigr)^m \leq \dfrac{4r^2}{n} \ \ \text{for all}\ \ n \geq 2r^2.
\end{eqnarray}
From \eqref{lem3eqn5} and \eqref{lem3eqn6}, it follows that for $n \geq 2r^2$,
\begin{eqnarray}\label{lem3eqn7}
H_r &\geq & \dfrac{\pi}{2}\Bigl(\dfrac{1}{2}\Bigr)_{r-1}\dfrac{1}{n^{r-\frac{1}{2}}}-\dfrac{\pi}{2}\Bigl(\dfrac{1}{2}\Bigr)_{r-1}\dfrac{4r^2}{n^{r+\frac{1}{2}}}-\dfrac{(r-1)!}{n^r}\nonumber\\
&>& \dfrac{\pi}{2}\Bigl(\dfrac{1}{2}\Bigr)_{r-1}\dfrac{1}{n^{r-\frac{1}{2}}}-\Bigl((r-1)!+2\pi r^2\Bigl(\dfrac{1}{2}\Bigr)_{r-1}\Bigr)\dfrac{1}{n^r}.
\end{eqnarray}
This finishes the estimation of the lower bound for $H_r$.\\
For the upper bound estimation of $H_r$, we start with \eqref{lem3eqn4} in the following way
\begin{eqnarray}\label{lem3eqn8}
H_r &\leq & \dfrac{C(r)}{n^{r-\frac{1}{2}}}-\dfrac{(r-1)!}{(n+r)^r}+\sum_{k=1}^{\infty}\dfrac{1}{k\pi^k}\Bigl(\dfrac{k}{2}\Bigr)_r \dfrac{1}{n^{r+\frac{k}{2}}}\nonumber\\
& < & \dfrac{C(r)}{n^{r-\frac{1}{2}}}-\dfrac{(r-1)!}{(2n)^r}+\sum_{k=1}^{\infty}\dfrac{1}{k\pi^k}\Bigl(\dfrac{k}{2}\Bigr)_r \dfrac{1}{n^{r+\frac{k}{2}}}\ \ \Bigl(\text{since},\ \dfrac{1}{(n+r)^r}> \dfrac{1}{(2n)^r}\ \forall\ n>r\Bigr)\nonumber\\
&=& \dfrac{C(r)}{n^{r-\frac{1}{2}}}-\dfrac{(r-1)!}{(2n)^r}+\dfrac{1}{n^{r+\frac{1}{2}}}\sum_{k=0}^{2r-2}\dfrac{1}{(k+1)\pi^{k+1}}\Bigl(\dfrac{k+1}{2}\Bigr)_r \dfrac{1}{\sqrt{n}^k}+\dfrac{1}{n^{r+\frac{1}{2}}}\sum_{k=2r}^{\infty}\dfrac{1}{k\pi^k}\Bigl(\dfrac{k}{2}\Bigr)_r \dfrac{1}{\sqrt{n}^{k-1}}\nonumber\\
& \leq & \dfrac{C(r)}{n^{r-\frac{1}{2}}}-\dfrac{(r-1)!}{(2n)^r}+\dfrac{1}{n^{r+\frac{1}{2}}}\underset{:=\widehat{C_2}(r)}{\underbrace{\sum_{k=0}^{2r-2}\dfrac{1}{(k+1)\pi^{k+1}}\Bigl(\dfrac{k+1}{2}\Bigr)_r \dfrac{1}{r^k}}}+\dfrac{r}{n^{r+\frac{1}{2}}}\underset{:=S(r)}{\underbrace{\sum_{k=2r}^{\infty}\dfrac{1}{k\pi^k}\Bigl(\dfrac{k}{2}\Bigr)_r \dfrac{1}{r^k}}}\\
& & \hspace{9 cm} \Bigl(\text{since},\ \dfrac{1}{\sqrt{n}^k} \leq \dfrac{1}{r^k} \forall\ n \geq r^2\Bigr).\nonumber
\end{eqnarray}
In order to estimate the infinite series $S(r)$, we need to give an upper bound of $\Bigl(\dfrac{k}{2}\Bigr)_r$ by rewriting as
\begin{equation*}
\Bigl(\dfrac{k}{2}\Bigr)_r=\Bigl(\dfrac{k}{2}\Bigr)^r \prod_{i=0}^{r-1}\Bigl(1+\dfrac{2i}{k}\Bigr):=\Bigl(\dfrac{k}{2}\Bigr)^r P(r,k).
\end{equation*}
Now,
\begin{equation}\label{lem3eqn9}
\log P(r,k) =\sum_{i=0}^{r-1}\log \Bigl(1+\dfrac{2i}{k}\Bigr)<\sum_{i=0}^{r-1}\dfrac{2i}{k}=\dfrac{r(r-1)}{k} \ \implies P(r,k) < e^{\frac{r(r-1)}{k}}.
\end{equation}
Using \eqref{lem3eqn9}, we obtain
\begin{equation}\label{lem3eqn10}
S(r) < \sum_{k=2r}^{\infty}\dfrac{1}{k\pi^k}\Bigl(\dfrac{k}{2}\Bigr)^r e^{\frac{r(r-1)}{k}} \dfrac{1}{r^k} \leq \dfrac{e^{\frac{r-1}{2}}}{2^r}\sum_{k=2r}^{\infty}\dfrac{k^{r-1}}{(\pi r)^k}\ \ \Bigl(\text{since}, e^{\frac{r(r-1)}{k}} \leq e^{\frac{r-1}{2}} \forall\ k \geq 2r\Bigr).
\end{equation}
Moreover, $k^{r-1}< r^{k}$ for all $r \geq 2$ and $k \geq 2r$. To observe this fact, we first note that to prove $k^{r-1}< r^{k}$, it is equivalent to show 
\begin{equation}\label{lem3eqn11}
\dfrac{r-1}{\log r}< \dfrac{k}{\log k}.
\end{equation}
Define $f(x):=\dfrac{x}{\log x}$ and observe that $f(x)$ is strictly increasing for all $x>e$. As $k \geq 2r \geq 4 >e$, it follows that $f(k)>f(2r)$ and the fact that $f(2r) > \dfrac{r-1}{\log r}$ for $r \geq 2$, we conclude \eqref{lem3eqn11}.\\
Applying \eqref{lem3eqn11} in \eqref{lem3eqn10}, we get
\begin{equation}\label{lem3eqn12}
S(r) < \dfrac{e^{\frac{r-1}{2}}}{2^r}\sum_{k=2r}^{\infty}\dfrac{1}{\pi^k} =\dfrac{\pi}{\sqrt{e}(\pi-1)}\Bigl(\dfrac{\sqrt{e}}{2 \ r^2}\Bigr)^r < \dfrac{1}{10^r}.
\end{equation} 
Hence, by \eqref{lem3eqn12} and \eqref{lem3eqn8}, we obtain for all $n \geq r^2$,
\begin{equation}\label{lem3eqn13}
H_r < \dfrac{C(r)}{n^{r-\frac{1}{2}}}-\dfrac{(r-1)!}{2^r\ n^r}+\dfrac{\widehat{C_2}(r)}{n^{r+\frac{1}{2}}}+\dfrac{r}{10^r\ n^{r+\frac{1}{2}}}=\dfrac{C(r)}{n^{r-\frac{1}{2}}}-\dfrac{(r-1)!}{2^r\ n^r}+ \underset{=C_2(r)}{\underbrace{\Bigl(\widehat{C_2}(r)+\dfrac{r}{10^r}\Bigr)}}\dfrac{1}{n^{r+\frac{1}{2}}}.
\end{equation}

\end{proof}
\section{Proof of Theorem \ref{mainresult1} and \ref{mainresult2} }\label{sec3}

\emph{Proof of Theorem \ref{mainresult1}}:
Applying \eqref{lem2eqn1} and \eqref{lem1eqn1} in \eqref{eqn8}, we have for $n \geq 85$,
\begin{equation}\label{rslt1eqn1}
L^{(1)}(n)-\dfrac{1}{n^2}<\Delta \log \p(n) < U^{(1)}(n)+\dfrac{1}{n^2}.
\end{equation}
It is straightforward to show that for $n \geq 457$,
\begin{equation}\label{rslt1eqn2}
-\dfrac{3}{2(n+1)}+\dfrac{\pi}{2\sqrt{n}(\widehat{\mu}(n)-1)}+ \dfrac{1}{n^2}< -\dfrac{\pi^2}{10n}
\end{equation}
and therefore
\begin{equation}\label{rslt1eqn3}
U^{(1)}(n)+\dfrac{1}{n^2} <\dfrac{\pi}{2\sqrt{n}}-\dfrac{\pi^2}{10n}.
\end{equation}
Define $c_n:=\dfrac{\pi}{2\sqrt{n}}-\dfrac{\pi^2}{10n}$ and $d_n:=\dfrac{\pi}{2\sqrt{n}}+\dfrac{\pi^2}{40n}$. Observe that $c_n<1$ for $n \geq 1$ and $d_n<1$ for $n \geq 3$ and consequently for $n \geq 3$,
\begin{equation}\label{rslt1eqn4}
c_n < d_n-\dfrac{d^2_n}{2}+\dfrac{d^3_n}{3}-\dfrac{d^4_n}{4}< \log (1+d_n)
\end{equation}
since, $\log(1+x)>x-\dfrac{x^2}{2}+\dfrac{x^3}{3}-\dfrac{x^4}{4}\ \text{for}\ x>0$. Invoking \eqref{rslt1eqn3} and \eqref{rslt1eqn4} in \eqref{rslt1eqn1}, we get for $n \geq 457$,
\begin{equation}\label{rslt1eqn5}
\Delta \log \p(n) < \log \Bigl(1+\dfrac{\pi}{2\sqrt{n}}+\dfrac{\pi^2}{40n}\Bigr).
\end{equation}
Similarly as before, it can be readily shown that for $n \geq 79$,
\begin{equation}\label{rslt1eqn6}
L^{(1)}(n)-\dfrac{1}{n^2}> \dfrac{\pi}{2\sqrt{n}}-\dfrac{\pi^2}{8n}+\dfrac{\pi^3}{24n^{3/2}}
\end{equation} 
and 
\begin{equation}\label{rslt1eqn7}
\dfrac{\pi}{2\sqrt{n}}-\dfrac{\pi^2}{8n}+\dfrac{\pi^3}{24n^{3/2}} > \log \Bigl(1+\dfrac{\pi}{2\sqrt{n}}\Bigr)
\end{equation}
as $\log(1+x)<x-\dfrac{x^2}{2}+\dfrac{x^3}{3}$ for $x>0$. Applying \eqref{rslt1eqn6} and \eqref{rslt1eqn7} into \eqref{rslt1eqn1}, it follows that for $n \geq 85$,
\begin{equation}\label{rslt1eqn8}
\Delta \log \p(n) > \log \Bigl(1+\dfrac{\pi}{2\sqrt{n}}\Bigr).
\end{equation} 
Equations \eqref{rslt1eqn5} and \eqref{rslt1eqn8} conclude the proof of Theorem \ref{mainresult1} except for $26 \leq n \leq 456$, which we confirm by numerical checking in Mathematica. 
\qed

\emph{Proof of Theorem \ref{mainresult2}}:
Applying \eqref{lem1eqn1} and \eqref{lem3eqn1} to the lower bound of \eqref{eqn8}, it follows that for $n \geq \max \{N_1(r), 2r^2\}$,
\begin{eqnarray}\label{rslt2eqn1}
(-1)^{r-1}\Delta^r \log \p(n) &>& \dfrac{\pi}{2}\Bigl(\dfrac{1}{2}\Bigr)_{r-1}\dfrac{1}{n^{r-\frac{1}{2}}}-\Bigl((r-1)!+2\pi r^2\Bigl(\dfrac{1}{2}\Bigr)_{r-1}\Bigr)\dfrac{1}{n^r}-\dfrac{1}{n^{r+1}}\nonumber\\
& > & \dfrac{\pi}{2}\Bigl(\dfrac{1}{2}\Bigr)_{r-1}\dfrac{1}{n^{r-\frac{1}{2}}}-\Bigl(1+(r-1)!+2\pi r^2\Bigl(\dfrac{1}{2}\Bigr)_{r-1}\Bigr)\dfrac{1}{n^r}\nonumber\\
&=& \dfrac{C(r)}{n^{r-\frac{1}{2}}}-\dfrac{1+C_1(r)}{n^r}.
\end{eqnarray} 
Following \eqref{def6}, $N_2(r)=\Biggl \lceil \Biggl(\dfrac{1+C_1(r)}{C(r)}\Biggr)^2\Biggr\rceil$. Then for all $n \geq \max \{N_1(r),2r^2,N_2(r)\}$, it follows that 
\begin{equation}\label{rslt2eqn2}
(-1)^{r-1}\Delta^r \log \p(n) > \log \Biggl(1+\dfrac{C(r)}{n^{r-\frac{1}{2}}}-\dfrac{1+C_1(r)}{n^r}\Biggr)>0.
\end{equation}
For the upper bound estimation, putting \eqref{lem1eqn1} and \eqref{lem3eqn1} together into the upper bound of \eqref{eqn8}, it follows that for $n \geq \max \{N_1(r), 2r^2\}$,
\begin{eqnarray}\label{rslt2eqn3}
(-1)^{r-1}\Delta^r \log \p(n) &<& \dfrac{C(r)}{n^{r-\frac{1}{2}}}-\dfrac{(r-1)!}{2^r\ n^r}+\dfrac{C_2(r)}{n^{r+\frac{1}{2}}}+\dfrac{1}{n^{r+1}}\nonumber\\
& <& \dfrac{C(r)}{n^{r-\frac{1}{2}}}-\dfrac{(r-1)!}{2^r\ n^r}+\dfrac{C_2(r)+1}{n^{r+\frac{1}{2}}}.
\end{eqnarray}
Next, our goal is to show 
\begin{equation*}
-\dfrac{(r-1)!}{2^r\ n^r}+\dfrac{C_2(r)+1}{n^{r+\frac{1}{2}}} < -\dfrac{C^2(r)}{2\ n^{2r-1}} 
\end{equation*}
which is equivalent to
\begin{equation}\label{rslt2eqn4}
\dfrac{C^2(r)}{2}< n^{r-1}\Biggl[\dfrac{(r-1)!}{2^r}-\dfrac{C_2(r)+1}{\sqrt{n}}\Biggr].
\end{equation}
Note that for all $n \geq \Biggl \lceil \Biggl(\dfrac{2^{r+1}\Bigl(C_2(r)+1\Bigr)}{(r-1)!}\Biggr)^2\Biggr \rceil$, $\dfrac{(r-1)!}{2^{r+1}}-\dfrac{C_2(r)+1}{\sqrt{n}}>0$ and therefore
\begin{equation}
n^{r-1}\Biggl[\dfrac{(r-1)!}{2^r}-\dfrac{C_2(r)+1}{\sqrt{n}}\Biggr]= n^{r-1}\Biggl[\dfrac{(r-1)!}{2^{r+1}}+\dfrac{(r-1)!}{2^{r+1}}-\dfrac{C_2(r)+1}{\sqrt{n}}\Biggr]>n^{r-1} \dfrac{(r-1)!}{2^{r+1}}.
\end{equation}
Hence, to prove \eqref{rslt2eqn4}, it is sufficient to prove
\begin{equation}
n^{r-1} \dfrac{(r-1)!}{2^{r+1}} > \dfrac{C^2(r)}{2} \ \text{which holds for all}\ \ n \geq \Biggl\lceil \sqrt[r-1]{\Biggl(\dfrac{2^r C^2(r)}{(r-1)!}\Biggr)} \Biggr\rceil.
\end{equation}
Recall that 
$$N_3(r)=\max \Biggl\{N_1(r),2r^2, \Biggl \lceil \Biggl(\dfrac{2^{r+1}\Bigl(C_2(r)+1\Bigr)}{(r-1)!}\Biggr)^2\Biggr \rceil, \Biggl\lceil \sqrt[r-1]{\Biggl(\dfrac{2^r C^2(r)}{(r-1)!}\Biggr)} \Biggr\rceil \Biggr\}\ \ (\text{cf.}\  \eqref{def7})$$.
From \eqref{rslt2eqn3} and \eqref{rslt2eqn4}, it follows that for $n \geq N_3(r)$,
\begin{equation}\label{rslt2eqn5}
(-1)^{r-1}\Delta^r \log \p(n) < \dfrac{C(r)}{n^{r-\frac{1}{2}}}-\dfrac{C^2(r)}{2\ n^{2r-1}}< \log \Bigl(1+\dfrac{C(r)}{n^{r-1/2}}\Bigr).
\end{equation}
Equation \eqref{rslt2eqn2} and \eqref{rslt2eqn5} together imply that for $n \geq \max \{N_2(r),N_3(r)\}=N(r)$, \eqref{mainresult2eqn} holds. 
\qed

\begin{center}
	\textbf{Acknowledgements}
\end{center}
 The author would like to express sincere gratitude to her advisor Prof. Manuel Kauers for his valuable suggestions on the paper. The author was supported by the Austrian Science Fund (FWF) grant W1214-N15, project DK13.

\end{document}